\UseRawInputEncoding

\documentclass[oneside, 11pt]{amsart} 

\usepackage{amsmath,amsthm,amssymb,epic}
\usepackage{eqlist,eqparbox}
\usepackage{amsfonts}
\usepackage{latexsym}
\usepackage[2emode]{psfrag}
\usepackage{amsthm}
\usepackage{amsmath}
\usepackage[all]{xy}

\newcommand{\Title}[1]{\bigskip\bigskip\centerline{\bf #1}\bigskip}
\newcommand{\Author}[1]{\medskip\centerline{ \it #1}}

\newcommand{\Affiliation}[1]{\medskip\centerline{#1}}
\newcommand{\Email}[1]{\medskip\centerline{#1}\bigskip}

\begin{document}

\newcommand{\N}{\mbox {$\mathbb N $}}
\newcommand{\Z}{\mbox {$\mathbb Z $}}
\newcommand{\Q}{\mbox {$\mathbb Q $}}
\newcommand{\R}{\mbox {$\mathbb R $}}
\newcommand{\lo }{\longrightarrow }
\newcommand{\ul}{\underleftarrow }
\newcommand{\rl}{\underrightarrow }
\newcommand{\rs }{\rightsquigarrow }
\newcommand{\ra }{\rightarrow } 
\newcommand{\dd }{\rightsquigarrow } 
\newcommand{\rars }{\Leftrightarrow }
\newcommand{\ol }{\overline }
\newcommand{\la }{\langle }
\newcommand{\tr }{\triangle }
\newcommand{\xr }{\xrightarrow }
\newcommand{\de }{\delta }
\newcommand{\pa }{\partial }
\newcommand{\LR }{\Longleftrightarrow }
\newcommand{\Ri }{\Rightarrow }
\newcommand{\va }{\varphi }
\newcommand{\Den}{{\rm Den}\,}
\newcommand{\Ker}{{\rm Ker}\,}
\newcommand{\Reg}{{\rm Reg}\,}
\newcommand{\Fix}{{\rm Fix}\,}
\newcommand{\Sup}{{\rm Sup}\,}
\newcommand{\Inf}{{\rm Inf}\,}
\newcommand{\Img}{{\rm Im}\,}
\newcommand{\Id}{{\rm Id}\,}
\newcommand{\ord}{{\rm ord}\,}

\newtheorem{theorem}{Theorem}[section]
\newtheorem{lemma}[theorem]{Lemma}
\newtheorem{proposition}[theorem]{Proposition}
\newtheorem{corollary}[theorem]{Corollary}
\newtheorem{definition}[theorem]{Definition}
\newtheorem{example}[theorem]{Example}
\newtheorem{examples}[theorem]{Examples}
\newtheorem{xca}[theorem]{Exercise}
\theoremstyle{remark}
\newtheorem{remark}[theorem]{Remark}
\newtheorem{remarks}[theorem]{Remarks}
\numberwithin{equation}{section}

\def\leftmark{L.C. Ciungu}

\Title{IMPLICATIVE-ORTHOMODULAR LATTICES} 
\title[Implicative-orthomodular lattices]{}
                                                                           
\Author{\textbf{LAVINIA CORINA CIUNGU}}
\Affiliation{Department of Mathematics} 
\Affiliation{St Francis College}
\Affiliation{180 Remsen Street, Brooklyn Heights, NY 11201-4398, USA}
\Email{lciungu@sfc.edu}

\begin{abstract} 
Based on implicative involutive BE algebras, we redefine the orthomodular lattices, by introducing the 
notion of implicative-orthomodular lattices, and we study their properties. 
We characterize these algebras, proving that the implicative-orthomodular lattices are quantum-Wajsberg algebras.  
We also define and characterize the implicative-modular algebras as a subclass of implicative-orthomodular 
lattices. 
The orthomodular softlattices and orthomodular widelattices are also redefined, by introducing the notions of implicative-orthomodular softlattices and implicative-orthomodular widelattices.
Finally, we prove that the implicative-orthomodular softlattices are equivalent to implicative-orthomodular 
lattices and that the implicative-orthomodular widelattices are special cases of quantum-Wajsberg algebras. \\

\textbf{AMS classification (2020):} 06C15, 03G25, 06A06, 81P10
\end{abstract}

\maketitle

\section{Introduction} 

Beside the classical and non-classical logics, there exist the quantum logics.
As algebraic structures connected with quantum logics we mention the following algebras: bounded involutive lattices,  De Morgan algebras, ortholattices, orthomodular lattices, MV algebras, quantum-MV algebras. 
Orthomodular lattices (particular ortholattices) generalize the Boolean algebras.
They have arisen in the study of quantum logic, that is, the logic which supports quantum mechanics and which does not conform to classical logic (see \cite{Ior35}). \\
The connections between algebras of logic/algebras and quantum algebras were clarified by A. Iorgulescu in 
\cite{Ior30, Ior31, Ior35}: it was proved that the quantum algebras belong, in fact, to the ``world" of involutive unital commutative magmas. 
These connections were established by redefining equivalently the bounded involutive 
lattices and De Morgan algebras as involutive m-MEL algebras and the ortholattices, the MV, the quantum MV and 
the Boolean algebras as involutive m-BE algebras, verifying some properties, and then putting all of them on
the involutive ``big map". \\
The quantum-MV algebras (or QMV algebras) were introduced by R. Giuntini in \cite{Giunt1} 
as non-lattice generalizations of MV algebras and as non-idempotent generalizations of orthomodular lattices. 
These structures were intensively studied by R. Giuntini (\cite{Giunt2, Giunt3, Giunt4, Giunt5, Giunt6}), 
A. Dvure\v censkij and S. Pulmannov\'a (\cite{DvPu}), R. Giuntini and S. Pulmannov\'a (\cite{Giunt7}) and by 
A. Iorgulescu in \cite{Ior30, Ior31, Ior32, Ior33, Ior34, Ior35, Ior37}. 
An extensive study on the orthomodular structures as quantum logics can be found in \cite{Ptak}. \\ 
Based on the involutive m-BE algebras, A. Iorgulescu redefined the quantum-MV algebras, and she generalized 
these algebras, by introducing three new algebras: orthomodular, pre-MV and meta-MV algebras. 
The implicative-ortholattices (also particular involutive BE algebras, introduced and studied in 
\cite{Ior30, Ior35}) are definitionally equivalent to ortholattices (redefined as particular involutive m-BE algebras in \cite{Ior30, Ior35}), and the implicative-Boolean algebras (introduced by A. Iorgulescu in 2009, also as particular involutive BE algebras, namely as particular involutive BCK algebras) are definitionally equivalent to Boolean algebras (redefined as particular involutive m-BE algebras, namely as particular (involutive) m-BCK algebras in \cite{Ior35}). 
The orthomodular lattices were also intensively studied by A. Iorgulescu (\cite{Ior32}) based on involutive m-BE algebras, and two generalizations of orthomodular lattices were introduced, orthomodular softlattices and 
orthomodular widelattices. \\
We redefined in \cite{Ciu78} the quantum-MV algebras as involutive BE algebras, by introducing and studying the 
notion of quantum-Wajsberg algebras. It was proved that the quantum-Wajsberg algebras are equivalent to quantum-MV algebras and that the Wajsberg algebras are both quantum-Wajsberg algebras and commutative quantum-B algebras 
(see also \cite{Ciu79, Ciu80}). 
We also redefined in \cite{Ciu81} the orthomodular algebras, by defining the implicative-orthomodular algebras, 
proving that the quantum-Wajsberg algebras are particular cases of these structures. 
We characterized the implicative-orthomodular algebras and we gave conditions for implicative-orthomodular algebras 
to be quantum-Wajsberg algebras. 
We also introduced and studied in \cite{Ciu82} the pre-Wajsberg algebras and the meta-Wajsberg algebras as  generalizations of quantum-Wajsberg algebras. \\
In this paper, we define and study the implicative-orthomodular lattices, we give characterizations of implicative-orthomodular lattices, and we show that they are quantum-Wajsberg algebras and implicative-orthomodular algebras. 
We also introduce the implicative-modular algebras, and we prove that they form a proper subclass of implicative-orthomodular lattices. 
Based on involutive BE algebras, we redefine the orthomodular softlattices and orthomodular widelattices, 
by introducing the notions of implicative-orthomodular softlattices and implicative-orthomodular widelattices. 
We prove that the implicative-orthomodular softlattices are equivalent to implicative-orthomodular lattices, 
and that  the implicative-orthomodular widelattices are special cases of quantum-Wajsberg algebras.  
Additionally, we study and characterize the implicative involutive BE algebras, and based on these algebras 
we define the implicative-orthosoftlattices and implicative-orthowidelattices, and we investigate 
their relationships with the implicative-ortholattices.

$\vspace*{1mm}$

\section{Preliminaries}

In this section, we recall some basic notions and results regarding lattices, BE algebras and m-BE algebras, 
as well as  quantum-Wajsberg, implicative-orthomodular, pre-Wajsberg and meta-Wajsberg algebras. 

\begin{definition} \label{iol-10} \emph{(\cite{PadRud}) (see also \cite{Ior30})} \\
\emph{
$(1)$ A \emph{lattice} is an algebra $(X,\wedge,\vee)$ such that the following properties hold, for all $x,y,z\in X$:\\
$(L_1)$ $x\wedge x=x$ and $x\vee x=x$,                  $\hspace*{5.00cm}$ (idempotency) \\
$(L_2)$ $x\wedge y=y\wedge x$ and $x\vee y=y\vee x$,    $\hspace*{3.80cm}$ (commutativity) \\
$(L_3)$ $x\wedge (y\wedge z)=(x\wedge y)\wedge z$ and $x\vee (y\vee z)=(x\vee y)\vee z$, 
                                                                           $\hspace*{0.1cm}$ (associativity) \\ 
$(L_4)$ $x\wedge (x\vee y)=x$ and $x\vee (x\wedge y)=x$. $\hspace*{3.15cm}$ (absorption) \\
$(2)$ A \emph{bounded lattice} is an algebra $(X,\wedge,\vee,0,1)$ such that $(X,\wedge,\vee)$ is a lattice 
and the following holds for any $x\in X$: \\
$(L_5)$ $1\wedge x=x$ and $0\vee x=x$. \\               
$(3)$ An \emph{orthollattice} (OL for short) is an algebra $(X,\wedge,\vee,^{'},0,1)$ such that $(X,\wedge,\vee,0,1)$ 
is a bounded lattice and the unary operation $^{'}$ satisfies the following properties, for all $x,y\in X$: \\
$(L_6)$ $(x^{'})^{'}=x$,                                    $\hspace*{7.50cm}$ (double negation) \\
$(L_7)$ $(x\vee y)^{'}=(x^{'}\wedge y^{'})$ and $(x\wedge y)^{'}=(x^{'}\vee y^{'})$,     
                                                                           $\hspace*{1.80cm}$ (De Morgan laws) \\
$(L_8)$ $x\wedge x^{'}=0$,                         $\hspace*{7.35cm}$ (noncontradiction principle) \\
$(L_9)$ $x\vee x^{'}=1$.                           $\hspace*{7.30cm}$ (excluded midle principle) 
}
\end{definition}

\begin{remarks} \label{iol-10-10}
$(1)$ It suffices to postulate only one of the De Morgan laws. 
For instance, if the first De Morgan law holds, then by the double negation law we have: 
$(x\wedge y)^{'}=((x^{'})^{'}\wedge (y^{'})^{'})^{'}=((x^{'}\vee y^{'})^{'})^{'}=x^{'}\vee y^{'}$, i.e. the second De Morgan law is also satisfied. \\
$(2)$ According to \cite{PadRud}, it was Dedekind who proved that the idempotency laws are consequences of 
absorption laws: 
$x\vee x=x\vee (x\wedge (x\vee x))=x$ and $x\wedge x=x\wedge (x\vee (x\wedge x))=x$. \\
$(3)$ For any $x\in X$, $1\wedge x=x$ $\rars$ $1\vee x=1$ and $0\vee x=x$ $\rars$ $0\wedge x=0$. \\
Indeed, if $1\wedge x=x$, then by absorption law, $1\vee x=1\vee (1\wedge x)=1$, and  
$1\vee x=1$ implies $1\wedge x=1\wedge (1\vee x)=1$. 
Similarly, from $0\vee x=x$ we get $0\wedge x=0\wedge (0\vee x)=0$ and 
$0\wedge x=0$ implies $0\vee x=0\vee (0\wedge x)=0$. \\
$(4)$ A. Iorgulescu proved in \cite[Th. 3.2]{Ior37} that the system of axioms $\{L_1, L_2, L_3,L_4\}$ 
for a lattice is equivalent to the system of axioms $\{L_1, L_2, L_3,L_4^{'}\}$, where: \\
$(L_4^{'})$ $x\wedge (x\vee x\vee y)=x$ and $x\vee (x\wedge x\wedge y)=x$. $\hspace*{1.80cm}$ 
(independent absorption) 
\end{remarks}

A. Iorgulescu introduced and studied in \cite{Ior37} the notions of softlattices, orthosoftlattices, widelattices and 
orthowidelattices as well as equivalent definitions for orthosoftlattices and orthowidelattices: \\
$-$ A \emph{softlattice} is an algebra $(X,\wedge,\vee)$ satisfying axioms $(L_1)$, $(L_2)$, $(L_3)$. 
A \emph{bounded softlattice} is an algebra $(X,\wedge,\vee,0,1)$ such that $(X,\wedge,\vee)$ is a softlattice and 
the elements $0$ and $1$ satisfy axiom $(L_5)$ (\cite[Def. 3.3]{Ior37}); \\
$-$ An \emph{orthosoftlattice} is an algebra $(X,\wedge,\vee,^*,0,1)$ such that 
$(X,\wedge,\vee,0,1)$ is a bounded softlattice satisfying axioms $(L_6)$-$(L_9)$ (\cite[Def. 5.1]{Ior37}); \\ 
$-$ A \emph{widelattice} is an algebra $(X,\wedge,\vee)$ satisfying axioms $(L_2)$, $(L_3)$, $(L_4^{'})$. 
A \emph{bounded widelattice} is an algebra $(X,\wedge,\vee,0,1)$ such that $(X,\wedge,\vee)$ is a widelattice and 
the elements $0$ and $1$ satisfy axiom $(L_5)$ (\cite[Def. 3.9]{Ior37}); \\
$-$ An \emph{orthowidelattice} (OWL for short) is an algebra $(X,\wedge,\vee,^*,0,1)$ such that 
$(X,\wedge,\vee,0,1)$ is a bounded widelattice satisfying axioms $(L_6)$-$(L_9)$ (\cite[Def. 5.6]{Ior37}). \\
An involutive m-BE algebra $(X,\odot,^*,1)$ is called: \\
$-$ an \emph{orthosoftlattice} (OSL for short), if it satisfies axiom: for any $x\in X$, \\
$(G)$ $x\odot x=x$ (\cite[Def. 5.3]{Ior37}); \\
$-$ an \emph{orthowidelattice} (OWL for short), if it satisfies axiom: for any $x,y\in X$, \\ 
$(m$-$Pabs$-$i)$ $x\odot (x\oplus x\oplus y)=x$ (\cite[Def. 5.8]{Ior37}). \\

\emph{BE algebras} were introduced in \cite{Kim1} as algebras $(X,\ra,1)$ of type $(2,0)$ satisfying the 
following conditions, for all $x,y,z\in X$: 
$(BE_1)$ $x\ra x=1;$ 
$(BE_2)$ $x\ra 1=1;$ 
$(BE_3)$ $1\ra x=x;$ 
$(BE_4)$ $x\ra (y\ra z)=y\ra (x\ra z)$. 
A relation $\le$ is defined on $X$ by $x\le y$ iff $x\ra y=1$. 
A BE algebra $X$ is \emph{bounded} if there exists $0\in X$ such that $0\le x$, for all $x\in X$. 
In a bounded BE algebra $(X,\ra,0,1)$ we define $x^*=x\ra 0$, for all $x\in X$. 
A bounded BE algebra $X$ is called \emph{involutive} if $x^{**}=x$, for any $x\in X$. 

\begin{lemma} \label{qbe-10} $\rm($\cite{Ciu78}$\rm)$ 
Let $(X,\ra,1)$ be a BE algebra. The following hold for all $x,y,z\in X$: \\
$(1)$ $x\ra (y\ra x)=1;$ 
$(2)$ $x\le (x\ra y)\ra y$. \\
If $X$ is bounded, then: \\
$(3)$ $x\ra y^*=y\ra x^*;$ 
$(4)$ $x\le x^{**}$. \\
If $X$ is involutive, then: \\
$(5)$ $x^*\ra y=y^*\ra x;$ 
$(6)$ $x^*\ra y^*=y\ra x;$ 
$(7)$ $(x\ra y)^*\ra z=x\ra (y^*\ra z);$ \\
$(8)$ $x\ra (y\ra z)=(x\ra y^*)^*\ra z;$    
$(9)$ $(x^*\ra y)^*\ra (x^*\ra y)=(x^*\ra x)^*\ra (y^*\ra y)$.  
\end{lemma}

\noindent
In a BE algebra $X$, we define the additional operation: \\
$\hspace*{3cm}$ $x\Cup y=(x\ra y)\ra y$. \\
If $X$ is involutive, we define the operations: \\
$\hspace*{3cm}$ $x\Cap y=((x^*\ra y^*)\ra y^*)^*$, 
                $x\odot y=(x\ra y^*)^*=(y\ra x^*)^*$, \\
and the relation $\le_Q$ by: \\
$\hspace*{3cm}$ $x\le_Q y$ iff $x=x\Cap y$. 

\begin{proposition} \label{qbe-20} $\rm($\cite{Ciu78}$\rm)$ Let $X$ be an involutive BE algebra. 
Then the following hold for all $x,y,z\in X$: \\
$(1)$ $x\le_Q y$ implies $x=y\Cap x$ and $y=x\Cup y;$ \\
$(2)$ $\le_Q$ is reflexive and antisymmetric; \\
$(3)$ $x\Cap y=(x^*\Cup y^*)^*$ and $x\Cup y=(x^*\Cap y^*)^*;$ \\ 
$(4)$ $x\le_Q y$ implies $x\le y;$ \\
$(5)$ $0\le_Q x \le_Q 1;$ \\
$(6)$ $0\Cap x=x\Cap 0=0$ and $1\Cap x=x\Cap 1=x;$ \\
$(7)$ $(x\Cap y)\ra z=(y\ra x)\ra (y\ra z);$ \\
$(8)$ $z\ra (x\Cup y)=(x\ra y)\ra (z\ra y);$ \\
$(9)$ $x\Cap y\le x,y\le x\Cup y;$ \\
$(10)$ $x\Cap (y\Cap x)=y\Cap x$ and $x\Cap (x\Cap y)=x\Cap y$. 
\end{proposition}

\begin{proposition} \label{qbe-30} $\rm($\cite{Ciu81}$\rm)$ Let $X$ be an involutive BE algebra. 
Then the following hold for all $x,y,z\in X$: \\
$(1)$ $x, y\le_Q z$ and $z\ra x=z\ra y$ imply $x=y;$ \emph{(cancellation law)} \\   
$(2)$ $(x\ra (y\ra z))\ra x^*=((y\ra z)\Cap x)^*;$ \\
$(3)$ $x\ra ((y\ra x^*)^*\Cup z)=y\Cup (x\ra z);$ \\
$(4)$ $((y\ra x)\Cap z)\ra x=y\Cup (z\ra x);$ \\
$(5)$ $x\le_Q y$ implies $(y\ra x)\odot y=x;$ \\ 
$(6)$ $x\ra (z\odot y^*)=((z\ra y)\odot x)^*;$ \\
$(7)$ $(x\Cup y)\Cap y=y$ and $(x\Cap y)\Cup y=y;$ \\
$(8)$ $(x\ra (x\ra y)^*)^*=y\Cap x;$ \\
$(9)$ $(x\Cap (y\Cap z))^*=((z\ra x)\Cap (z\ra y))\ra z^*;$ \\ 
$(10)$ $(x\Cap y)^*\ra (y\ra x)^*=y\Cup (y\ra x)^*$. 
\end{proposition}

A \emph{(left-)m-BE algebra} (\cite{Ior30}) is an algebra $(X,\odot,^{*},1)$ of type $(2,1,0)$ satisfying the 
following properties, for all $x,y,z\in X$:  
(PU) $1\odot x=x=x\odot 1;$ 
(Pcomm) $x\odot y=y\odot x;$ 
(Pass) $x\odot (y\odot z)=(x\odot y)\odot z;$  
(m-L) $x\odot 0=0;$ 
(m-Re) $x\odot x^{*}=0$, 
where $0:=1^*$. 

\begin{remark} \label{qmv-30-10} 
According to \cite[Cor. 17.1.3]{Ior35}, the involutive (left-)BE algebras $(X,\ra,^*,1)$ are definitionally equivalent to involutive (left-)m-BE algebras $(X,\odot,^*,1)$, by the mutually inverse transformations 
(\cite{Ior30, Ior35}): \\ 
$\hspace*{3cm}$ $\Phi:$\hspace*{0.2cm}$ x\odot y:=(x\ra y^*)^*$ $\hspace*{0.1cm}$ and  
                $\hspace*{0.1cm}$ $\Psi:$\hspace*{0.2cm}$ x\ra y:=(x\odot y^*)^*$. 
\end{remark}

\noindent 
A \emph{quantum-Wajsberg algebra} (\emph{QW algebra, for short}) (\cite{Ciu78}) $(X,\ra,^*,1)$ is an 
involutive BE algebra $(X,\ra,^*,1)$ satisfying the following condition: for all $x,y,z\in X$, \\
(QW) $x\ra ((x\Cap y)\Cap (z\Cap x))=(x\ra y)\Cap (x\ra z)$. \\
It was proved in \cite{Ciu78} that condition (QW) is equivalent to the following conditions: \\
$(QW_1)$ $x\ra (x\Cap y)=x\ra y;$ \\ 
$(QW_2)$ $x\ra (y\Cap (z\Cap x))=(x\ra y)\Cap (x\ra z)$. \\
For more details on quantum-Wajsberg algebras we refer the reader to \cite{Ciu78}. \\
An involutive BE algebra $X$ is called (\cite{Ciu82}): \\
- a \emph{pre-Wajsberg algebra} (preW algebra for short), if it verifies $(QW_1);$ \\
- an \emph{implicative-orthomodular algebra} (IOM algebra for short), if it verifies $(QW_2);$ \\
- a \emph{meta-Wajsberg algebra} (metaW algebra for short), if it verifies: \\
$(QW_3)$ $(x\Cap y)\ra (y\Cap x)=1$. \\
Condition $(QW_3)$ is equivalent to condition: for all $x,y\in X$, \\
$(QW_3^{'})$ $(x\Cup y)\ra (y\Cup x)=1$. \\
According to \cite{Ciu78}, a quantum-Wajsberg algebra satisfies condition $(QW_3)$. 

\noindent
Denoting by $\mathbf{QW}$, $\mathbf{IOM}$, $\mathbf{preW}$ and $\mathbf{metaW}$ the classes of quantum-Wajsberg, 
implicative-orthomodular, pre-Wajsberg and meta-Wajsberg algebras, respectively, we have: \\
$\hspace*{2cm}$ $\mathbf{QW}=\mathbf{preW}\bigcap \mathbf{IOM}=\mathbf{metaW}\bigcap \mathbf{IOM}$, 
$\mathbf{preW}\subset \mathbf{metaW}$. \\

$\vspace*{1mm}$

\section{Implicative BE algebras}

In this section, we define the implicative BE algebras, and we study and characterize the implicative involutive 
BE algebras.  
We also introduce the notions of implicative-orthosoftlattices and implicative-orthowidelattices and we investigate 
their relationships with the implicative-ortholattices. 

\begin{definition} \label{iol-20} 
\emph{
A BE algebra is called \emph{implicative} if it satisfies the following condition: for all $x,y\in X$, \\
$(Impl)$ $(x\ra y)\ra x=x$. 
}
\end{definition}

\begin{lemma} \label{iol-30} Let $(X,\ra,^*,1)$ be an implicative involutive BE algebra. 
Then $X$ verifies the following axioms: for all $x,y\in X$, \\
$(iG)$       $x^*\ra x=x$, or equivalently, $x\ra x^*=x^*;$ \\
$(Iabs$-$i)$ $(x\ra (x\ra y))\ra x=x;$ \\
$(Pimpl)$    $x\ra (x\ra y)=x\ra y$.    
\end{lemma}
\begin{proof} 
Let $X$ be an implicative involutive BE algebra, so that it satisfies axiom $(Impl)$. \\
$(iG)$ Taking $y:=0$ in $(Impl)$ we have $x^*\ra x=x$, and replacing $x$ by $x^*$ in this identity, 
we get $x\ra x^*=x^*$, that is $(iG)$. \\ 
$(Iabs$-$i)$ By $(iG)$ and $(Impl)$, we have $((x\ra x^*)^*\ra y)\ra x=x$. 
Since by Lemma \ref{qbe-10}$(7)$, $(x\ra x^*)^*\ra y=x\ra (x\ra y)$, it follows that $X$ verifies $(Iabs$-$i)$. \\
$(Pimpl)$ Using $(iG)$ and Lemma \ref{qbe-10}$(6)$, we have: \\ 
$\hspace*{2cm}$ $x\ra (x\ra y)=x\ra (y^*\ra x^*)=y^*\ra (x\ra x^*)=y^*\ra x^*=x\ra y$. 
\end{proof}

\begin{lemma} \label{iol-30-05} Let $(X,\ra,^*,1)$ be an implicative involutive BE algebra. 
The following hold for all $x,y\in X$: \\
$(1)$ $x\ra (y\ra x)^*=x^*;$ \\
$(2)$ $(y\ra x^*)\ra x=x;$ \\
$(3)$ $(x\ra (x^*\ra y)^*)^*=x;$ \\
$(4)$ $x^*\ra (x\ra y^*)^*=x;$ \\
$(5)$ $x\ra (y\ra x^*)=y\ra x^*$.    
\end{lemma}
\begin{proof} 
$(1)$ By $(Impl)$, we get: \\
$\hspace*{1cm}$ $(x\ra y)\ra x=x$ iff $x^*\ra (x\ra y)^*=x$ iff $x^*\ra (y^*\ra x^*)^*=x$, \\
and replacing $x$ by $x^*$ and $y$ by $y^*$, it follows that $x\ra (y\ra x)^*=x^*$. \\
$(2)$ Applying $(1)$, we have $(y\ra x^*)\ra x=x^*\ra (y\ra x^*)^*=x$. \\
$(3)$ By $(Impl)$, $(x^*\ra y)\ra x^*=x^*$, hence $x\ra (x^*\ra y)^*=x^*$, so that $(x\ra (x^*\ra y)^*)^*=x$. \\
$(4)$ Using $(Impl)$, we get $x^*\ra (x\ra y^*)^*=(x\ra y^*)\ra x=x$. \\
$(5)$ Applying $(iG)$, we get $x\ra (y\ra x^*)=y\ra (x\ra x^*)=y\ra x^*$.  
\end{proof}

\begin{corollary} \label{iol-30-10} Let $(X,\ra,^*,1)$ be an involutive BE algebra. The following are equivalent: \\
$(a)$ $X$ is implicative; \\
$(b)$ $X$ verifies axioms $(iG)$ and $(Iabs$-$i)$; \\ 
$(c)$ $X$ verifies axioms $(Pimpl)$ and $(Iabs$-$i)$. 
\end{corollary}
\begin{proof}
If $X$ is implicative, then by Lemma \ref{iol-30} it verifies axioms $(iG)$, $(Pimpl)$ and $(Iabs$-$i)$. 
Hence $(a)\Rightarrow (b)$ and $(a)\Rightarrow (c)$. 
Conversely, we see that, by Lemma \ref{qbe-10}$(7)$, axiom $(Iabs$-$i)$ is equivalent to axiom: for all $x,y\in X$, \\
$(Iabs$-$i^{'})$ $((x\ra x^*)^*\ra y)\ra x=x$. \\ 
Since by $(iG)$, $x\ra x^*=x^*$, it follows from $(Iabs$-$i^{'})$ that $(x\ra y)\ra x=x$. 
Hence $X$ satisfies axiom $(Impl)$, that is $(b)\Rightarrow (a)$.  
Moreover, if $X$ satisfies $(Iabs$-$i)$ and $(Pimpl)$, we have $x=(x\ra (x\ra y))\ra x=(x\ra y)\ra x$. 
It follows that $X$ verifies $(Impl)$, so that $(c)\Rightarrow (a)$.  
\end{proof}

\begin{proposition} \label{iol-30-20} Let $(X\ra,^*,1)$ be an involutive BE algebra satisfying axiom $(QW_1)$.  
Then axioms $(Impl)$ and $(Pimpl)$ are equivalent. 
\end{proposition}
\begin{proof} 
By Lemma \ref{iol-30}, axiom $(Impl)$ implies axiom $(Pimpl)$. 
Conversely, assume that $X$ satisfies $(Pimpl)$ and let $x,y\in X$. Since $x\ra (x\ra y)=x\ra y$, we have: \\
$\hspace*{2.00cm}$ $x^*\Cap (x\ra y)^*=(x\Cup (x\ra y))^*=((x\ra (x\ra y))\ra (x\ra y))^*$ \\
$\hspace*{4.50cm}$ $=((x\ra y)\ra (x\ra y))^*=1^*=0$. \\
Using $(QW_1)$ and Lemma \ref{qbe-10}$(5)$, we get: \\
$\hspace*{2.00cm}$ $x=(x^*)^*=x^*\ra 0=x^*\ra (x^*\Cap (x\ra y)^*)=x^*\ra (x\ra y)^*=(x\ra y)\ra x$. \\ 
Hence axiom $(Impl)$ is satisfied. 
\end{proof}

\begin{proposition} \label{gioml-90-10} Let $(X,\ra,^*,1)$ be an involutive BE algebra. 
If $X$ satisfies axioms $(QW_1)$ and $(iG)$, then $X$ is implicative. 
\end{proposition}
\begin{proof}
Let $X$ be an involutive BE algebra satisfying axioms $(QW_1)$ and $(iG)$, and let $x,y\in X$. 
Using $(iG)$, we have: $x\ra (x\ra y)=x\ra (y^*\ra x^*)=y^*\ra (x\ra x^*)=y^*\ra x^*=x\ra y$. 
Hence $X$ satisfies condition $(Pimpl)$. 
Moreover, we have succesively: \\
$\hspace*{2.00cm}$ $x\ra (x\Cap y)=x\ra y$ (by $(QW_1)$), \\
$\hspace*{2.00cm}$ $(x\Cap y)^*\ra x^*=y^*\ra x^*$ (Lemma \ref{qbe-10}$(6)$), \\
$\hspace*{2.00cm}$ $((x^*\ra y^*)\ra y^*)\ra x^*=y^*\ra x^*$, \\
$\hspace*{2.00cm}$ $((x^*\ra (x^*\ra y^*))\ra (x^*\ra y^*))\ra x^*=(x^*\ra y^*)\ra x^*$ \\
$\hspace*{8.00cm}$                           (replacing $y$ by $(x^*\ra y^*)^*$), \\
$\hspace*{2.00cm}$ $((x^*\ra y^*)\ra (x^*\ra y^*))\ra x^*=(x^*\ra y^*)\ra x^*$ (by $(Pimpl)$), \\
$\hspace*{2.00cm}$ $1\ra x^*=(x^*\ra y^*)\ra x^*$, \\
$\hspace*{2.00cm}$ $x^*=(x^*\ra y^*)\ra x^*$. \\
Hence $(x^*\ra y^*)\ra x^*=x^*$, and replacing $x,y$ by $x^*,y^*$, respectively, we get $(x\ra y)\ra x=x$. 
It follows that $X$ satisfies axiom $(Impl)$, so that it is implicative. 
\end{proof}

\begin{definition} \label{iol-40} \emph{(\cite[Def. 3.29]{Ior30})} 
\emph{
An implicative involutive BE algebra $(X,\ra,^*,1)$ is called an \emph{implicative-ortholattice} (IOL for short). 
}
\end{definition}

\begin{proposition} \label{iol-50} The implicative-ortholattices are defitionally equivalent to ortholattices. 
\end{proposition}
\begin{proof}
Let $(X,\wedge,\vee,^{'},0,1)$ be an ortholattice. 
Define $x\ra y:=(x\wedge y^{'})^{'}$, and we can see that $x\ra y=x^{'}\vee y$ and 
$x^*:=x\ra 0=(x\wedge 0^{'})^{'}=(x\wedge 1)^{'}=x^{'}$, for all $x,y\in X$. 
Axioms $(BE_1)$-$(BE_4)$ follow immediately from $(L_1)$-$(L_9)$. 
For example, to prove $(BE_4)$, using $(L_3)$, $(L_2)$ we have: \\
$\hspace*{2.00cm}$ $x\ra (y\ra z)=x\ra (y\wedge z^{'})^{'}=(x\wedge (y\wedge z^{'}))^{'}
                                 =x^{'}\vee (y\wedge z^{'})^{'}$ \\
$\hspace*{4.25cm}$ $=x^{'}\vee (y^{'}\vee z)=y^{'}\vee (x^{'}\vee z)=(y\wedge (x\wedge z^{'}))^{'}$ \\
$\hspace*{4.25cm}$ $=y\ra (x\wedge z^{'})^{'}=y\ra (x\ra z)$. \\
Moreover, by $(L_2)$ and $(L_4)$, we have: \\
$\hspace*{1.00cm}$  $(x\ra y)\ra x=(x\wedge y^{'})\ra x=((x\wedge y^{'})^{'}\wedge x^{'})^{'}
                                  =(x\wedge y^{'})\vee x=x\vee (x\wedge y^{'})=x$, \\
that is axiom $(Impl)$ is verified. 
Since the double negation axiom holds, $(X,\ra,^*,1)$ is an implicative involutive BE algebra. \\
Conversely, let $(X,\ra,^*,1)$ be an implicative involutive BE algebra, and define: 
$0:=1^*$, $x\wedge y:=(x\ra y^*)^*$ and $x\vee y:=x^*\ra y=(x^*\wedge y^*)^*$. 
We have $x^{'}=(x\wedge 1)^{'}=(x\wedge 0^{'})^{'}=x\ra 0=x^*$. 
Axioms $(L_1)$-$(L_9)$ can be easily proved using axioms $(BE_1)$-$(BE_4)$, $(Impl)$, involution and 
Lemma \ref{qbe-10}.  
For example, to prove $(L_4)$, using Lemma \ref{qbe-10}$(5)$ and $(Impl)$ we have: 
$x\wedge (x\vee y)=(x\ra (x\vee y)^*)^*=(x\ra (x^*\ra y)^*)=((x^*\ra y)\ra x^*)^*=(x^*)^*=x$. 
Hence $(X,\wedge,\vee,^*,0,1)$ is an ortholattice. 
\end{proof}

\noindent
In what follows we define the notions of implicative-orthosoftlattices and implicative-orthowidelattices. 

\begin{definition} \label{iol-90}  
\emph{
An involutive BE algebra $(X,\ra,^*,1)$ is called: \\
$-$ an \emph{implicative-orthosoftlattice} (IOSL for short) if it satisfies axiom: \\ 
$(iG)$ $x^*\ra x=x$, or equivalently, $x\ra x^*=x^*;$ \\
$-$ an \emph{implicative-orthowidelattice} (IOWL for short) if it satisfies axiom: \\ 
$(Iabs$-$i)$ $(x\ra (x\ra y))\ra x=x$.   
}
\end{definition}

\begin{proposition} \label{iol-100} The following hold: \\
$(1)$ the implicative-orthosoftlattices are defitionally equivalent to orthosoftlattices; \\
$(2)$ the implicative-orthowidelattices are defitionally equivalent to orthowidelattices. 
\end{proposition}
\begin{proof}
By Remark \ref{qmv-30-10}, the involutive BE algebras are defitionally equivalent to involutive m-BE algebras. 
Using the mutually inverse transformations (\cite{Ior30, Ior35}): \\ 
$\hspace*{3cm}$ $\Phi:$\hspace*{0.2cm}$ x\odot y:=(x\ra y^*)^*$ $\hspace*{0.1cm}$ and  
                $\hspace*{0.1cm}$ $\Psi:$\hspace*{0.2cm}$ x\ra y:=(x\odot y^*)^*$, \\
and the relation $x\oplus y=(x^*\odot y^*)^*$, we can easily show that axiom $(iG)$ is equivalent to $(G)$, and 
axiom $(Iabs$-$i)$ is equivalent to $(m$-$Pabs$-$i)$. 
\end{proof}

\begin{remark} \label{iol-110}
Denote by $\mathbf{IOL}$, $\mathbf{ISOL}$ and $\mathbf{IWOL}$ the classes of implicative-ortholattice, 
implicative-orthosoftlattice and implicative-orthowideolattice, respectively. 
As a consequence of Corollary \ref{iol-30-10}, we have $\mathbf{IOL}=\mathbf{ISOL}\bigcap \mathbf{IWOL}$. 
\end{remark}

$\vspace*{1mm}$

\section{Implicative-orthomodular lattices} 

Based on implicative involutive BE algebras, we define and study the implicative-orthomodular lattices, 
proving that they are defitionally equivalent to orthomodular lattices. 
We give certain characterizations of implicative-orthomodular lattices, and we show that they are quantum-Wajsberg algebras and implicative-orthomodular algebras. 
The implicative-modular algebras are also introduced, and is proved that they form a proper subclass of implicative-orthomodular lattices. 

\begin{definition} \label{ioml-10} \emph{(\cite{PadRud}, \cite[Def. 3.1]{Ior32})} 
\emph{
An \emph{orthomodular lattice} (OML for short) is an ortholattice $(X,\wedge,\vee,^{'},0,1)$ verifying: 
for all $x,y\in X$,\\
$(OM)$ $(x\wedge y)\vee ((x\wedge y)^{'}\wedge x)=x$.  
}
\end{definition}

\begin{remark} \label{ioml-10-10}
$(1)$ Axiom $(OM)$ is equivalent to axiom: for all $x,y\in X$, \\ 
$(OM^{'})$ $x\le y$ implies $x\vee (x^{'}\wedge y)=y$ (where $x\le y$ iff $x=x\wedge y$) (see \cite{Burris}). \\
$(2)$ It was proved in \cite[Th. 2.3.9]{DvPu} that any orthomodular lattice $(X,\wedge,\vee,^{'},0,1)$ is a quantum-MV algebra $(X,\oplus,^{*},0,1)$, taking $\oplus$ as as the supremum $\vee$ and $^{*}$ as the orthocomplement $^{'}$. 
Conversely, if an ortholattice $(X,\wedge,\vee,^{'},0,1)$ determines a quantum-MV algebra taking $\oplus=\vee$ and 
$^*=^{'}$, then $X$ is an orthomodular lattice. 
\end{remark}

\begin{definition} \label{ioml-20} 
\emph{
An \emph{implicative-orthomodular lattice} (IOML for short) is an implicative involutive BE algebra $(X,\ra,^*,1)$  
verifying: for all $x,y,z\in X$:\\
($QW_2$) $x\ra (y\Cap (z\Cap x))=(x\ra y)\Cap (x\ra z)$.  
}
\end{definition}

\noindent
In the other words, an implicative-orthomodular lattice is an implicative-ortholattice satisfying 
condition $(QW_2)$ or, equivalently, an involutive BE algebra satisfying conditions $(Impl)$ and $(QW_2)$.  
Obviously, the implicative-orthomodular lattices are implicative-orthomodular algebras. 

\begin{lemma} \label{ioml-30} Let $X$ be an involutive BE algebra. 
The following are equivalent for all $x,y\in X$: \\
$(IOM)$ $x\Cap (y\ra x)=x;$ \\
$(IOM^{'})$ $x\Cap (x^*\ra y)=x;$ $\hspace*{5cm}$ \\
$(IOM^{''})$ $x\Cup (x\ra y)^*=x$. 
\end{lemma}
\begin{proof}
The proof is straightforward. 
\end{proof}

\begin{proposition} \label{ioml-40} Let $X$ be an involutive BE algebra satisfying the equivalent conditions 
$(IOM)$, $(IOM^{'})$, $(IOM^{''})$. Then the following hold for all $x,y,z\in X$: \\
$(1)$ $x\Cap (y\Cup x)=x$ and $x\Cup (y\Cap x)=x$. \\
If $x\le_Q y$, then: \\
$(2)$ $y\Cup x=y$ and $y^*\le_Q x^*;$ \\ 
$(3$ $y\ra z\le_Q x\ra z$ and $z\ra x\le_Q z\ra y;$ \\
$(4)$ $x\Cap z\le_Q y\Cap z$ and $x\Cup z\le_Q y\Cup z$. 
\end{proposition}
\begin{proof}
$(1)$ Using $(IOM)$, we get: \\
$\hspace*{1.50cm}$ $x\Cap (y\Cup x)=x\Cap ((y\ra x)\ra x)=x;$ \\ 
$\hspace*{1.50cm}$ $x\Cup (y\Cap x)=x\Cup((y^*\ra x^*)\ra x^*)^*=(x^*\Cap ((y\ra x^*)\ra x^*))^*=(x^*)^*=x$. \\
$(2)$ From $x=x\Cap y$, using $(1)$ we get $y\Cup x=y\Cup (x\Cap y)=y$. 
Hence $y^*=(y\Cup x)^*=y^*\Cap x^*$, so that $y^*\le_Q x^*$. \\
$(3)$ Since by $(IOM^{'})$, $y\ra z\le_Q (y\ra z)^*\ra (y\ra x)^*$,  we have: \\
$\hspace*{2.00cm}$ $(y\ra z)\Cap (x\ra z)=(y\ra z)\Cap ((x\Cap y)\ra z)$ \\
$\hspace*{5.15cm}$ $=(y\ra z)\Cap(((x^*\ra y^*)\ra y^*)^*\ra z)$ \\
$\hspace*{5.15cm}$ $=(y\ra z)\Cap (z^*\ra ((x^*\ra y^*)\ra y^*))$ \\
$\hspace*{5.15cm}$ $=(y\ra z)\Cap ((x^*\ra y^*)\ra (z^*\ra y^*))$ \\
$\hspace*{5.15cm}$ $=(y\ra z)\Cap ((y\ra x)\ra (y\ra z))$ \\
$\hspace*{5.15cm}$ $=(y\ra z)\Cap ((y\ra z)^*\ra (y\ra x)^*)$ \\
$\hspace*{5.15cm}$ $=y\ra z$. \\
Thus that $y\ra z\le_Q x\ra z$. 
Applying $(2)$, $x\le_Q y$ implies $y=y\Cup x$, and by $(IOM^{'})$, 
$z\ra x\le_Q (z\ra x)^*\ra (y\ra x)^*$, hence: \\
$\hspace*{2.00cm}$ $(z\ra x)\Cap (z\ra y)=(z\ra x)\Cap (z\ra (y\Cup x))$ \\
$\hspace*{5.15cm}$ $=(z\ra x)\Cap (z\ra ((y\ra x)\ra x))$ \\
$\hspace*{5.15cm}$ $=(z\ra x)\Cap ((y\ra x)\ra (z\ra x))$ \\
$\hspace*{5.15cm}$ $=(z\ra x)\Cap ((z\ra x)^*\ra (y\ra x)^*)$ \\
$\hspace*{5.15cm}$ $=z\ra x$. \\
It follows that $z\ra x\le_Q z\ra y$. \\
$(4)$ Since $x\le_Q y$ implies $y^*\le_Q x^*$, using $(3)$ we get $x^*\ra z^*\le_Q y^*\ra z^*$ and 
$(y^*\ra z^*)\ra z^*\le_Q (x^*\ra z^*)\ra z^*$. 
Hence $((x^*\ra z^*)\ra z^*)^*\le_Q ((y^*\ra z^*)\ra z^*)^*$, that is $x\Cap z\le_Q y\Cap z$. 
Using again $(3)$, we also get $y\ra z\le_Q x\ra z$ and $(x\ra z)\ra z\le_Q (y\ra z)\ra z$. 
Thus $x\Cup z\le_Q y\Cup z$. 
\end{proof}

\begin{proposition} \label{ioml-50} Let $X$ be an involutive BE algebra satisfying the equivalent conditions 
$(IOM)$, $(IOM^{'})$, $(IOM^{''})$. Then the following hold for all $x,y,z\in X$: \\
$(1)$ $x\ra (y\Cap x)=x\ra y;$ \\  
$(2)$ $(x\Cup y)\ra (x\ra y)^*=y^*;$ \\ 
$(3)$ $x\Cap ((y\ra x)\Cap (z\ra x))=x;$ \\
$(4)$ $x\Cup ((y^*\ra x^*)^*\Cup (z^*\ra x^*))=x$.  
\end{proposition}
\begin{proof}
$(1)$ Applying (IOM), it follows that: \\
$\hspace*{2.00cm}$ $x^*=(x\Cap (y\ra x))^*=(x^*\ra (y\ra x)^*)\ra (y\ra x)^*$ \\
$\hspace*{2.50cm}$ $=((y\ra x)\ra x)\ra (y\ra x)^*=(y\Cup x)\ra (y\ra x)^*$ \\
$\hspace*{2.50cm}$ $=(y\ra x)\ra (y\Cup x)^*=(x^*\ra y^*)\ra (y^*\Cap x^*)$. \\
Replacing $x$ by $x^*$ and $y$ by $y^*$, we get $(x\ra y)\ra (y\Cap x)=x$. \\
$(2)$ Using (IOM), we have $y=y\Cap (x\ra y)=((y^*\ra (x\ra y)^*)\ra (x\ra y)^*)^*$. 
It follows that $y^*=(y^*\ra (x\ra y)^*)\ra (x\ra y)^*=((x\ra y)\ra y)\ra (x\ra y)^*=(x\Cup y)\ra (x\ra y)^*$. \\
$(3)$ Since by (IOM), $x\le_Q y\ra x, z\ra x$, applying Proposition \ref{ioml-40}$(4)$, we have: 
$x=x\Cap (z\ra x)\le_Q (y\ra x)\Cap (z\ra x)$, hence $x\Cap ((y\ra x)\Cap (z\ra x))=x$. \\
$(4)$ As a consequence of $(3)$, we have $x^*\Cup ((y\ra x)^*\Cup (z\ra x)^*))=x^*$. 
Replacing $x, y, z$ with $x^*, y^*, z^*$, respectively, we get $x\Cup ((y^*\ra x^*)^*\Cup (z^*\ra x^*))=x$. 
\end{proof}

\begin{theorem} \label{ioml-60} 
Let $(X,\ra,^*,1)$ be an involutive BE algebra. Then the following are equivalent: \\
$(a)$ $X$ is an implicative-orthomodular lattice; \\
$(b)$ $X$ is an implicative involutive BE satisfying one of conditions $(IOM)$, $(IOM^{'})$, $(IOM^{''})$. 
\end{theorem}
\begin{proof}
$(a)\Rightarrow (b)$ Let $X$ be an implicative-orthomodular lattice, so that $X$ is an involutive BE algebra 
satisfying axioms $(Impl)$ and $(QW_2)$. 
Applying $(QW_2)$ for $y:=0$, we get $x^*=x^*\Cap (x\ra z)$, and replacing $x$ with $x^*$ and $z$ with $y$,  
it follows that $x=x\Cap (x^*\ra y)$. Hence $X$ satisfies condition $(IOM^{'})$, and according to 
Lemma \ref{ioml-30}, conditions $(IOM)$ and $(IOM^{''})$ are also satisfied. \\  
$(b)\Rightarrow (a)$ Suppose that $X$ satisfies $(Impl)$ and the equivalent conditions $(IOM)$, $(IOM^{'})$, $(IOM^{''})$. 
It follows that: \\
$\hspace*{0.50cm}$ $x\ra (y\Cap (z\Cap x))=x\ra (y\Cap (z^*\Cup x^*)^*)$ \\
$\hspace*{3.50cm}$ $=x\ra (y^*\Cup (z^*\Cup x^*))^*$ \\ 
$\hspace*{3.50cm}$ $=x\ra (y^*\Cup ((z^*\ra x^*)\ra x^*))^*$ \\
$\hspace*{3.50cm}$ $=x\ra (y^*\Cup (x\ra (z^*\ra x^*)^*))^*$ (Lemma \ref{qbe-10}$(5)$) \\
$\hspace*{3.50cm}$ $=x\ra (x\ra ((y^*\ra x^*)^*\Cup (z^*\ra x^*)^*))^*$ (Prop. \ref{qbe-30}$(3)$) \\
$\hspace*{3.50cm}$ $=(x\Cup ((y^*\ra x^*)^*\Cup (z^*\ra x^*)^*))\ra (x\ra ((y^*\ra x^*)^*\Cup (z^*\ra x^*)^*))^*$ \\
                                                      $\hspace*{12.00cm}$ (Prop. \ref{ioml-50}$(4)$) \\
$\hspace*{3.50cm}$ $=((y^*\ra x^*)^*\Cup (z^*\ra x^*)^*)^*$ (Prop. \ref{ioml-50}$(2)$) \\
$\hspace*{3.50cm}$ $=(y^*\ra x^*)\Cap (z^*\ra x^*)=(x\ra y)\Cap (x\ra z)$. \\
Hence axiom $(QW_2)$ is satisfied, so that $X$ is an implicative-orthomodular lattice. 
\end{proof}

\begin{proposition} \label{ioml-70} 
The orthomodular lattices $(X,\wedge,\vee,^{'},0,1)$ are defitionally equivalent to implicative-orthomodular 
lattices $(X,\ra,^*,1)$, by the mutually inverse transformations \\
$\hspace*{3cm}$ $\varphi:$\hspace*{0.2cm}$ x\ra y:=(x\wedge y^{'})^{'}$ $\hspace*{0.1cm}$ and  
                $\hspace*{0.1cm}$ $\psi:$\hspace*{0.2cm}$ x\wedge y:=(x\ra y^*)^*$, \\
and the relation $x\vee y:=(x^{'}\wedge y^{'})^{'}=x^*\ra y$.                   
\end{proposition}
\begin{proof} 
Using the transformations $\varphi$ and $\psi$, we have: \\
$\hspace*{2cm}$ $x^*=x\ra 0=(x\wedge 0^{'})^{'}=(x\wedge 1)^{'}=x^{'}$ and \\ 
$\hspace*{2cm}$ $x^{'}=(x\wedge 1)^{'}=(x\wedge 0^{'})^{'}=x\ra 0=x^*$, \\
for all $x\in X$. 
According to Proposition \ref{iol-50}, the ortholattices are defitionally equivalent to implicative-ortholattices.
We prove that axioms $(OM)$ and $(QW_2)$ are equivalent. 
By mutually inverse transformations $\varphi$ and $\psi$, we get: \\
$\hspace*{2.00cm}$ $(x\wedge y)\vee ((x\wedge y)^*\wedge x)=x$ iff \\
$\hspace*{2.00cm}$ $(x\wedge y)^*\ra ((x\wedge y)^*\ra x^*)^*=x$ iff \\
$\hspace*{2.00cm}$ $((x\wedge y)^*\ra x^*)\ra (x\wedge y)=x$ iff \\
$\hspace*{2.00cm}$ $((x\ra y^*)\ra x^*)\ra (x\ra y^*)^*=x$ iff \\
$\hspace*{2.00cm}$ $(x\ra (x\ra y^*)^*)\ra (x\ra y^*)^*=x$ iff \\
$\hspace*{2.00cm}$ $x\Cup (x\ra y^*)^*=x$ iff \\
$\hspace*{2.00cm}$ $(IOM^{''})$ iff \\
$\hspace*{2.00cm}$ $(QW_2)$ (Theorem \ref{ioml-60}). 
\end{proof}

\begin{proposition} \label{ioml-80} 
Let $(X,\ra,^*,1)$ be an implicative-orthomodular lattice and let $x,y\in X$. 
Then $x\le_Q y$ if and only if $y\ra x^*=x^*$. 
\end{proposition}
\begin{proof}
Assume $x\le_Q y$, that is $x\Cap y=x$. Since by Lemma \ref{iol-30}, $(Impl)$ implies $(iG)$, we have 
$y\ra y^*=y^*$. It follows that: \\
$\hspace*{2.00cm}$ $y\ra x^*=y\ra (x\Cap y)^*=y\ra ((x^*\ra y^*)\ra y^*)$ \\
$\hspace*{3.35cm}$ $=(x^*\ra y^*)\ra (y\ra y^*)=(x^*\ra y^*)\ra y^*=(x\Cap y)^*=x^*$. \\
Conversely, consider $x,y\in X$ such that $y\ra x^*=x^*$. 
Then we get: \\
$\hspace*{2.00cm}$ $x\Cap y=((x^*\ra y^*)\ra y^*)^*=(((y\ra x^*)\ra y^*)\ra y^*)^*$ \\
$\hspace*{2.95cm}$ $=(((x\ra y^*)\ra y^*)\ra y^*)^*=((x\Cup y^*)\ra y^*)^*$ \\
$\hspace*{2.95cm}$ $=(y\ra (x\Cup y^*)^*)^*=(y\ra (x^*\Cap y))^*=(y\ra x^*)^*$  (Prop. \ref{ioml-50}$(1)$) \\ 
$\hspace*{2.95cm}$ $=(x^*)^*=x$. \\
Hence $x\Cap y=x$, that is $x\le_Q y$. 
\end{proof}

\noindent
We prove the following result using an idea from \cite{Ciu82}. 
 
\begin{theorem} \label{ioml-90} Let $(X,\ra,^*,1)$ be an implicative involutive BE algebra. 
The following are equivalent: \\
$(a)$ $X$ satisfies axiom $(QW_1);$ \\
$(b)$ $X$ satisfies axiom $(QW_2);$ \\
$(c)$ $X$ satisfies axiom $(QW)$. 
\end{theorem}
\begin{proof}
By hypothesis, $X$ verifies condition $(Impl)$. \\
$(a)\Rightarrow (b)$ Assume that $X$ satisfies axiom $(QW_1)$, and we get succesively: \\
$\hspace*{2.00cm}$ $x\ra (x\Cap y)=x\ra y$ (by $(QW_1)$), \\
$\hspace*{2.00cm}$ $x\ra ((x^*\ra y^*)\ra y^*)^*=x\ra y$, \\ 
$\hspace*{2.00cm}$ $x\ra (y\ra (x^*\ra y^*)^*)^*=x\ra y$ (Lemma  \ref{qbe-10}$(6)$), \\ 
$\hspace*{2.00cm}$ $x\ra (y\ra (y\ra x)^*)^*=x\ra y$ (Lemma \ref{qbe-10}$(6)$) \\ 
$\hspace*{2.00cm}$ $x\ra ((y\ra x^*)\ra ((y\ra x^*)\ra x)^*)^*=x\ra (y\ra x^*)$ (replacing $y$ by $y\ra x^*$), \\ 
$\hspace*{2.00cm}$ $x\ra ((y\ra x^*)\ra x^*)^*=y\ra x^*$ (Lemma \ref{iol-30-05}$(2),(5)$), \\
$\hspace*{2.00cm}$ $x\ra (x\ra (y\ra x^*)^*)=y\ra x^*$ (Lemma  \ref{qbe-10}$(6)$), \\
$\hspace*{2.00cm}$ $(x\ra y^*)\ra ((x\ra y^*)\ra (x^*\ra (x\ra y^*)^*)^*)=x^*\ra (x\ra y^*)^*$, \\
$\hspace*{9.00cm}$ (replacing $x$ by $x\ra y^*$ and $y$ by $x^*$) \\
$\hspace*{2.00cm}$ $(x\ra y^*)\ra ((x\ra y^*)\ra x^*)^*=x$ (Lemma \ref{iol-30-05}$(4)$), \\
$\hspace*{2.00cm}$ $((x\ra y^*)\ra x^*)\ra (x\ra y^*)^*=x$ (Lemma  \ref{qbe-10}$(6)$), \\
$\hspace*{2.00cm}$ $(x\ra (x\ra y^*)^*)\ra (x\ra y^*)^*=x$ (Lemma  \ref{qbe-10}$(6)$), \\
$\hspace*{2.00cm}$ $x\Cup (x\ra y^*)^*=x$, \\
$\hspace*{2.00cm}$ $x\Cup (x\ra y)^*=x$ (replacing $y$ by $y^*$). \\
Hence $X$ satisfies condition $(IOM^{''})$, and by Theorem \ref{ioml-60}, it follows that 
$X$ satisfies axiom $(QW_2)$. \\
$(b)\Rightarrow (c)$ If $X$ satisfies axiom $(QW_2)$, then by Theorem\ref{ioml-60}, it satisfies conditions 
$(IOM)$, $(IOM^{'})$, $(IOM^{''})$. We have: \\
$\hspace*{2.00cm}$ $x\ra (x\Cap y)=x\ra ((x^*\ra y^*)\ra y^*)^*$ \\ 
$\hspace*{4.10cm}$ $=(x\ra (x^*\ra y)^*)^*\ra ((x^*\ra y^*)\ra y^*)^*$ (Lemma \ref{iol-30-05}$(3)$) \\ 
$\hspace*{4.10cm}$ $=x\ra ((x^*\ra y^*)\ra ((x^*\ra y^*)\ra y^*)^*)$ (Lemma \ref{qbe-10}$(7)$) \\
$\hspace*{4.10cm}$ $=x\ra (((x^*\ra y^*)\ra y^*)\ra (x^*\ra y^*)^*)$ (Lemma \ref{qbe-10}$(6)$) \\
$\hspace*{4.10cm}$ $=x\ra ((y\ra (x^*\ra y^*)^*)\ra (x^*\ra y^*)^*)$ (Lemma \ref{qbe-10}$(6)$) \\
$\hspace*{4.10cm}$ $=x\ra (y\Cup (x^*\ra y^*)^*)$ \\
$\hspace*{4.10cm}$ $=x\ra (y^*\Cap (x^*\ra y^*))^*$ \\
$\hspace*{4.10cm}$ $=x\ra (y^*)^*$ (by $(IOM)$) \\
$\hspace*{4.10cm}$ $=x\ra y$. \\
Hence $X$ verifies axiom $(QW_1)$. 
It follows that $X$ satisfies axioms $(QW_1)$ and $(QW_2)$, so that $X$ satisfies $(QW)$. \\
$(c)\Rightarrow (a)$ It is obvious, since $(QW)$ is equivalent to axioms $(QW_1)$ and $(QW_2)$.  
\end{proof}

\begin{corollary} \label{ioml-100} 
The implicative-orthomodular lattices are quantum-Wajsberg algebras, implicative-orthomodular algebras and 
pre-Wajsberg algebras.  
\end{corollary}

\begin{proposition} \label{ioml-110} 
The implicative-orthomodular lattices are meta-Wajsberg algebras.  
\end{proposition}
\begin{proof}
Let $(X,\ra,^*,1)$ be an implicative-orthomodular lattice, that is an implicative involutive BE algebra satisfying 
axiom $(QW_2)$. By Corollary \ref{ioml-100}, $X$ is a quantum-Wajsberg algebra, and it was proved in 
\cite{Ciu78} that a quantum-Wajsberg algebra satisfies axiom $(QW_3)$. 
It follows that $X$ is a meta-Wajsberg algebra.  
\end{proof}

\begin{remark} \label{ioml-120} 
If we denote by $\mathbf{IOML}$ the class of implicative-orthomodular lattices, then by Definition \ref{ioml-20},  Theorem \ref{ioml-90} and Proposition \ref{ioml-110}, we have: \\
$\hspace*{2cm}$ $\mathbf{IOML}=\mathbf{IOM}\bigcap \mathbf{IOL}=\mathbf{preW}\bigcap \mathbf{IOL}
=\mathbf{QW}\bigcap \mathbf{IOL}$, \\    
$\hspace*{2cm}$ $\mathbf{IOML}\subset \mathbf{metaW}$. 
\end{remark}

\noindent
The following example follows from \cite[Ex. 13.3.2]{Ior35}. 

\begin{example} \label{ioml-130}
Let $X=\{0,a,b,c,d,e,f,g,h,1\}$ and let $(X,\ra,^*,1)$ be the involutive BE algebra with $\ra$ and the corresponding 
operation $\Cap$ given in the following tables: \\
\[
\begin{array}{c|cccccccccc}
\ra & 0 & a & b & c & d & e & f & g & h & 1 \\ \hline
0   & 1 & 1 & 1 & 1 & 1 & 1 & 1 & 1 & 1 & 1 \\ 
a   & b & 1 & b & 1 & 1 & h & f & f & h & 1 \\ 
b   & a & a & 1 & 1 & 1 & a & 1 & a & 1 & 1 \\ 
c   & d & 1 & 1 & 1 & d & 1 & 1 & 1 & 1 & 1 \\
d   & c & 1 & 1 & c & 1 & 1 & 1 & 1 & 1 & 1 \\
e   & f & 1 & f & 1 & 1 & 1 & f & f & 1 & 1 \\
f   & e & a & h & 1 & 1 & e & 1 & a & h & 1 \\
g   & h & 1 & h & 1 & 1 & h & 1 & 1 & h & 1 \\
h   & g & a & f & 1 & 1 & a & f & g & 1 & 1 \\
1   & 0 & a & b & c & d & e & f & g & h & 1
\end{array}
\hspace{10mm}
\begin{array}{c|cccccccccc}
\Cap & 0 & a & b & c & d & e & f & g & h & 1 \\ \hline
0    & 0 & 0 & 0 & 0 & 0 & 0 & 0 & 0 & 0 & 0 \\ 
a    & 0 & a & 0 & c & d & e & g & g & e & a \\ 
b    & 0 & 0 & b & c & d & 0 & b & 0 & b & b \\ 
c    & 0 & a & b & c & 0 & e & f & g & h & c \\
d    & 0 & a & b & 0 & d & e & f & g & h & d \\
e    & 0 & e & 0 & c & d & e & 0 & 0 & e & e \\
f    & 0 & g & b & c & d & 0 & f & g & b & f \\
g    & 0 & g & 0 & c & d & 0 & g & g & 0 & g \\
h    & 0 & e & b & c & d & e & b & 0 & h & h \\
1    & 0 & a & b & c & d & e & f & g & h & 1
\end{array}
.
\]

We can see that $(X,\ra,^*,1)$ is an implicative involutive BE algebra satisfying axiom $(QW_2)$, hence it is  
an implicative-orthomodular lattice. 
One can check that $X$ verifies axioms $(QW_1)$ and $(QW_3)$, that is $X$ is also a quantum-Wajsberg algebra, 
a pre-Wajsberg algebra and a meta-Wajsberg algebra. 
\end{example}

\noindent
Recall the following notions from \cite{PadRud} (see also \cite{Ior32}). \\
A lattice $(X,\wedge,\vee)$ is \emph{modular}, if for all $x,y,z\in X$, \\
$(Wmod)$ $x\wedge (y\vee (x\wedge z))=(x\wedge y)\vee (x\wedge z)$, \\
and, dually, the dual lattice $(X,\vee,\wedge)$ is \emph{modular}, if for all $x,y,z\in X$, \\
$(Vmod)$ $x\vee (y\wedge (x\vee z))=(x\vee y)\wedge (x\vee z)$. \\
A \emph{modular ortholattice} is an ortholattice $(X,\wedge,\vee,^{'},0,1)$ whose lattice $(X,\wedge,\vee)$ is 
modular. It was proved in \cite{PadRud} that a modular ortholattice is an orthomodular lattice. \\

A. Iorgulescu presented in \cite{Ior32} an equivalent definition of modular ortholattices, and she introduced 
the notion of modular algebras. \\
A \emph{modular ortholattice} (\cite[Def. 3.19]{Ior32}) is an involutive m-BE algebra $(X,\odot,^*,1)$ verifying: for all $x,y,z\in X$, \\
$(m$-$Pimpl)$ $((x\odot y^*)^*\odot x^*)^*=x$, \\
$(Pmod)$ $x\odot (y\oplus (x\odot z))=(x\odot y)\oplus (x\odot z)$. \\
A \emph{modular algebra} (MOD algebra for short) (\cite[Def. 3.20]{Ior32}) is an involutive m-BE algebra $(X,\odot,^*,1)$ verifying axiom $(Pmod)$. 

Now, we define the notion of implicative-modular algebras. 

\begin{definition} \label{ioml-150} 
\emph{
An \emph{implicative-modular algebra} (IMOD algebra for short) is an involutive BE algebra $(X,\ra,^*,1)$ 
verifying: for all $x,y,z\in X$, \\
$(Imod)$ $(x\ra (y\ra (x\ra z)^*)^*)^*=(x\ra y)\ra (x\ra z)^*$.  
}
\end{definition}

\begin{remark} \label{ioml-150-10}
Axiom $(Imod)$ is equivalent to the following axiom: for all $x,y,z\in X$, \\
$(Imod^{'})$ $((z\ra x)\ra y)\ra x=((y\ra x)\ra (z\ra x)^*)^*$. \\
Indeed, applying Lemma \ref{qbe-10}$(6)$, we have the following equivalences: \\
$\hspace*{2cm}$ $x\ra (y\ra (x\ra z)^*)^*=((x\ra y)\ra (x\ra z)^*)^*$ (by $(Imod)$) iff \\
$\hspace*{2cm}$ $(y\ra (x\ra z)^*)\ra x^*=((y^*\ra x^*)\ra (z^*\ra x^*)^*)^*$ iff \\
$\hspace*{2cm}$ $((z^*\ra x^*)\ra y^*)\ra x^*=((y^*\ra x^*)\ra (z^*\ra x^*)^*)^*$ iff \\
$\hspace*{2cm}$ $((z\ra x)\ra y)\ra x=((y\ra x)\ra (z\ra x)^*)^*$ \\
$\hspace*{6cm}$ (replacing $x$,$y$,$z$ by $x^*$,$y^*$,$z^*$, respectively).
\end{remark}

\begin{remark} \label{ioml-150-20}
Using the mutually inverse transformations $\varphi$ and $\psi$ defined in the proof of Proposition \ref{ioml-70}, 
we can easily prove that axioms $(Imod)$ and $(Wmod)$ are equivalent. Hence the implicative-modular algebras  
are defitionally equivalent to modular lattices. 
Similary, using the transformations $\Phi$ and $\Psi$ from Proposition \ref{iol-100}, it follows that axioms $(Imod)$ 
and $(Pmod)$ are also equivalent. Thus the implicative-modular algebras are defitionally equivalent to 
modular algebras in the second definition. 
\end{remark}

\begin{proposition} \label{ioml-160} Let $(X,\ra,^*,1)$ be an implicative-modular algebra. The following hold 
for all $x,y,z\in X$: \\
$(1)$ $x\ra (y\ra (x\ra z)^*)^*=x\ra (z\ra (x\ra y)^*)^*;$ \\
$(2)$ $x\ra (y\ra (x\ra (x\ra y)^*)^*)^*=x^*;$ \\
$(3)$ $(x\ra y^*)\ra (x\ra y)^*=(x\ra (x^*\Cap y^*))^*;$ \\
$(4)$ $X$ is an implicative involutive BE algebra.    
\end{proposition}
\begin{proof}
$(1)$ Let $(X,\ra,^*,1)$ be an implicative-modular algebra, that is it is an involutive BE algebra 
satisfying axiom $(Imod)$. 
Applying twice axiom $(Imod)$, we get: \\
$\hspace*{2.00cm}$ $x\ra (y\ra (x\ra z)^*)^*=((x\ra y)\ra (x\ra z)^*)^*=((x\ra z)\ra (x\ra y)^*)^*$ \\
$\hspace*{5.65cm}$ $=x\ra (z\ra (x\ra y)^*)^*$. \\
$(2)$ Replacing $z$ by $(x\ra y)^*$ in $(1)$, we have: \\
$x\ra (y\ra (x\ra (x\ra y)^*)^*)^*=x\ra ((x\ra y)^*\ra (x\ra y)^*)^*=x\ra 1^*=x\ra 0=x^*$. \\
$(3)$ Replacing $y$ by $y^*$ and $z$ by $y$ in $(Imod)$, it follows that: \\
$\hspace*{2.00cm}$ $(x\ra y^*)\ra (x\ra y)^*=(x\ra (y^*\ra (x\ra y)^*)^*)^*=(x\ra ((x\ra y)\ra y)^*)^*$ \\
$\hspace*{5.75cm}$ $=(x\ra (x\Cup y)^*)^*=(x\ra (x^*\Cap y^*))^*$. \\
$(4)$ Taking $y:=0$ and $z:=y$ in axiom $(Imod)$, we get: \\
$\hspace*{2.00cm}$ $x^*\ra (x\ra y)^*=(x\ra (0\ra (x\ra y)^*)^*)^*=(x\ra 1^*)^*=x$. \\
It follows that $(x\ra y)\ra x=x$, that is $(Impl)$, so that $X$ is an implicative involutive BE algebra. 
\end{proof}

\begin{theorem} \label{ioml-170} Any implicative-modular algebra is an implicative-orthomodular lattice. 
\end{theorem}
\begin{proof}
Let $(X,\ra,^*,1)$ be an implicative-modular algebra, that is it is an involutive BE algebra 
satisfying axiom $(Imod)$. According to Proposition \ref{ioml-160}$(4)$, $X$ is implicative. 
Replacing $y$ by $(x\ra y)^*$ and $z$ by $y$ in condition $(Imod)$, we have: \\ 
$\hspace*{2.00cm}$ $x\Cup (x\ra y)^*=(x\ra (x\ra y)^*)\ra (x\ra y)^*$ \\
$\hspace*{4.25cm}$ $=(x\ra ((x\ra y)^*\ra (x\ra y)^*)^*)^*$ \\ 
$\hspace*{4.25cm}$ $=(x\ra 1^*)^*=x$. \\
Hence $X$ satisfies condition $(IOM^{''})$, and according to Theorem \ref{ioml-60}, it is an implicative-orthomodular lattice. 
\end{proof}

\begin{corollary} \label{ioml-180} 
The implicative-modular algebras are implicative-orthomodular algebras, quantum-Wajsberg algebras and 
pre-Wajsberg algebras.  
\end{corollary}
\begin{proof} It follows by Theorem \ref{ioml-170} and Corollary \ref{ioml-100}. 
\end{proof}

\begin{remark} \label{ioml-200} The converse of Theorem \ref{ioml-170} is not always true. 
Indeed, consider the implicative-orthomodular lattice $X$ from Example \ref{ioml-130}. 
Axiom $(Imod)$ is not satisfied for $x=a, y=c, z=e$, so that $X$ is not an implicative-modular algebra. 
\end{remark}

\noindent
The next example follows from \cite[Ex. 13.3.1]{Ior35}. 

\begin{example} \label{ioml-210} 
Let $X=\{0,a,b,c,d,1\}$ and let $(X,\ra,^*,1)$ be the involutive BE algebra 
with $\ra$ and the corresponding operation $\Cap$ given in the following tables:  
\[
\begin{array}{c|ccccccc}
\ra & 0 & a & b & c & d & 1 \\ \hline
0   & 1 & 1 & 1 & 1 & 1 & 1 \\ 
a   & b & 1 & b & 1 & 1 & 1 \\ 
b   & a & a & 1 & 1 & 1 & 1 \\ 
c   & d & 1 & 1 & 1 & d & 1 \\
d   & c & 1 & 1 & c & 1 & 1 \\
1   & 0 & a & b & c & d & 1
\end{array}
\hspace{10mm}
\begin{array}{c|ccccccc}
\Cap & 0 & a & b & c & d & 1 \\ \hline
0    & 0 & 0 & 0 & 0 & 0 & 0 \\ 
a    & 0 & a & 0 & c & d & a \\ 
b    & 0 & 0 & b & c & d & b \\ 
c    & 0 & a & b & c & 0 & c \\
d    & 0 & a & b & 0 & d & d \\
1    & 0 & a & b & c & d & 1
\end{array}
.
\]

We can see that $(X,\ra,^*,1)$ satisfies axiom $(Imod)$, hence it is an implicative-modular algebra. 
\end{example}

$\vspace*{1mm}$

\section{Implicative-orthomodular softlattices and widelattices}

Based on involutive BE algebras, we redefine the orthomodular softlattices and orthomodular widelattices, 
by introducing the notions of implicative-orthomodular softlattices and implicative-orthomodular widelattices. 
We prove that the implicative-orthomodular softlattices are equivalent to implicative-orthomodular lattices, 
and that the implicative-orthomodular widelattices are special cases of quantum-Wajsberg algebras. 
We also show that the implicative-orthomodular softlattices are defitionally equivalent to orthomodular 
softlattices, and the implicative-orthomodular widelattices are defitionally equivalent to orthomodular 
widelattices. \\
\noindent
Let $(X,\ra,^*,1)$ be an involutive BE algebra  verifying condition: \\ 
$(Pom)$ $(x\odot y)\oplus ((x\odot y)^*\odot x)=x$ or, equivalently, $x\Cup (x\odot y)=x$. \\
Then $X$ is called (\cite[Def. 4.2, 4.19]{Ior32}): \\ 
- an \emph{orthomodular softlattices} (OMSL for short) if it satisfies condition: \\ 
$(G)$ $x\odot x=x;$ \\  
- an \emph{orthomodular widelattices} (OMWL for short) if it satisfies condition: \\ 
$(m$-$Pabs$-$i)$ $x\odot (x\oplus x\oplus y)=x$.   

\begin{definition} \label{gioml-20} 
\emph{
Let Let $(X,\ra,^*,1)$ be an involutive BE algebra  verifying condition $(QW_2)$. Then $X$ is called: \\
- an \emph{implicative-orthomodular softlattices} (IOMSL for short) if it satisfies condition $(iG)$; \\
- an \emph{implicative-orthomodular widelattices} (IOMWL for short) if it satisfies condition $(Iabs$-$i)$.  
}
\end{definition} 

\begin{remark} \label{gioml-30}
Since by Theorem \ref{ioml-60}, condition $(QW_2)$ is equivalent to condition $(IOM^{''})$, 
then by the mutually inverse transformations (\cite{Ior30, Ior35}) \\ 
$\hspace*{3cm}$ $\Phi:$\hspace*{0.2cm}$ x\odot y:=(x\ra y^*)^*$ $\hspace*{0.1cm}$ and  
                $\hspace*{0.1cm}$ $\Psi:$\hspace*{0.2cm}$ x\ra y:=(x\odot y^*)^*$, \\
the following hold: \\
$(1)$ the implicative-orthomodular softlattices are defitionally equivalent to orthomodular softlattices; \\
$(2)$ the implicative-orthomodular widelattices are defitionally equivalent to orthomodular widelattices.              
\end{remark}

\begin{proposition} \label{gioml-90} 
An involutive BE algebra satisfying conditions $(QW_2)$ and $(iG)$ is implicative. 
\end{proposition}
\begin{proof}
Let $X$ be an involutive BE algebra satisfying conditions $(QW_2)$ and $(iG)$, and let $x,y\in X$. 
Since $X$ satisfies $(QW_2)$, then it satisfies $(IOM)$, $(IOM^{'})$ , $(IOM^{''})$. Then we have: \\
$\hspace*{2.00cm}$ $(x\ra y)\ra x=(x\ra y)\ra (x\Cup (x\ra y)^*)$ (by $(IOM^{''})$) \\ 
$\hspace*{4.35cm}$ $=(x\ra y)\ra ((x\ra (x\ra y)^*)\ra (x\ra y)^*)$ \\
$\hspace*{4.35cm}$ $=(x\ra (x\ra y)^*)\ra ((x\ra y)\ra (x\ra y)^*)$ \\
$\hspace*{4.35cm}$ $=(x\ra (x\ra y)^*)\ra (x\ra y)^*$ (by $(iG)$) \\
$\hspace*{4.35cm}$ $=x\Cup (x\ra y)^*=x$ (by $(IOM^{''})$). \\
Hence $(x\ra y)\ra x=x$, so that $X$ satisfies axiom $(Impl)$. 
Thus $X$ is an implicative involutive BE algebra.
\end{proof}

\begin{proposition} \label{gioml-40} Let $(X,\ra,^*,1)$ be an involutive BE algebra. 
If $X$ satisfies axioms $(QW_2)$ and $(iG)$, then $X$ satisfies axiom $(QW_1)$. 
\end{proposition}
\begin{proof}
Let $X$ be an involutive BE algebra satisfying axioms $(QW_2)$ and $(iG)$, and let $x,y\in X$. 
By Proposition \ref{gioml-90}, $X$ is implicative.
Since $X$ satisfies $(QW_2)$, then it satisfies $(IOM)$, $(IOM^{'})$, $(IOM^{''})$. It follows that: \\
$\hspace*{2.00cm}$ $x\ra y=x\ra (y\Cup (y\ra x)^*)$ (by $(IOM^{''})$) \\ 
$\hspace*{3.05cm}$ $=x\ra ((y\ra (y\ra x)^*)\ra (y\ra x)^*)$ \\
$\hspace*{3.05cm}$ $=(y\ra (y\ra x)^*)\ra (x\ra (y\ra x)^*)$  (by $(BE_4)$) \\
$\hspace*{3.05cm}$ $=(x\ra (y\ra x)^*)^*\ra (y\ra (y\ra x)^*)^*$ (Lemma \ref{qbe-10}$(6)$) \\
$\hspace*{3.05cm}$ $=((y\ra x)\ra x^*)^*\ra ((y\ra x)\ra y^*)^*$ (Lemma \ref{qbe-10}$(6)$) \\
$\hspace*{3.05cm}$ $=((x^*\ra y^*)\ra x^*)^*\ra ((x^*\ra y^*)\ra y^*)^*$ \\
$\hspace*{3.05cm}$ $=(x^*)^*\ra ((x^*\ra y^*)\ra y^*)^*$ (by $(Impl)$) \\
$\hspace*{3.05cm}$ $=x\ra (x\Cap y)$. \\
Thus $x\ra (x\Cap y)=x\ra y$, that is $X$ satisfies axiom $(QW_1)$. 
\end{proof}

\begin{proposition} \label{gioml-50} Let $(X,\ra,^*,1)$ be an involutive BE algebra. 
If $X$ satisfies axioms $(QW_1)$ and $(iG)$, then $X$ satisfies axiom $(QW_2)$. 
\end{proposition}
\begin{proof}
Let $X$ be an involutive BE algebra satisfying axioms $(QW_1)$ and $(iG)$, and let $x,y\in X$. 
By Proposition \ref{gioml-90-10}, $X$ is implicative. 
Then we have succesively: \\
$\hspace*{2.00cm}$ $x\ra (x\Cap y)=x\ra y$  (by $(QW_1)$), \\
$\hspace*{2.00cm}$ $x\ra (x\Cap (y\ra x^*))=x\ra (y\ra x^*)$ (replacing $y$ by $y\ra x^*$), \\
$\hspace*{2.00cm}$ $x\ra (x^*\Cup (y\ra x^*)^*)^*=y\ra (x\ra x^*)$ (by $(BE_4)$), \\
$\hspace*{2.00cm}$ $(x^*\Cup (y\ra x^*)^*)\ra x^*=y\ra x^*$ (Lemma \ref{qbe-10}$(6)$ and $(iG)$), \\
$\hspace*{2.00cm}$ $((x^*\ra (y\ra x^*)^*)\ra (y\ra x^*)^*)\ra x^*=x\ra y^*$, \\
$\hspace*{2.00cm}$ $(((y\ra x^*)\ra x)\ra (y\ra x^*)^*)\ra x^*=x\ra y^*$, \\
$\hspace*{2.00cm}$ $(((x\ra y^*)\ra x)\ra (y\ra x^*)^*)\ra x^*=x\ra y^*$, \\
$\hspace*{2.00cm}$ $(x\ra (y\ra x^*)^*)\ra x^*=x\ra y^*$ (by $(Impl)$), \\
$\hspace*{2.00cm}$ $((y\ra x^*)\ra x^*)\ra x^*=x\ra y^*$, \\
$\hspace*{2.00cm}$ $(y\Cup x^*)\ra x^*=x\ra y^*$, \\
$\hspace*{2.00cm}$ $x\ra (y\Cup x^*)^*=x\ra y^*$, \\
$\hspace*{2.00cm}$ $x\ra (y^*\Cap x)=x\ra y^*$, \\
$\hspace*{2.00cm}$ $(x\ra y^*)\ra (x\Cap (x\ra y^*))=(x\ra y^*)\ra x$ 
                                              (replacing $x$ by $x\ra y^*$ and $y$ by $x^*$), \\
$\hspace*{2.00cm}$ $(x\ra y^*)\ra (x^*\Cup (x\ra y^*)^*)^*=x$ (by $(Impl)$), \\
$\hspace*{2.00cm}$ $(x^*\Cup (x\ra y^*)^*)\ra (x\ra y^*)^*=x$, \\
$\hspace*{2.00cm}$ $((x^*\ra (x\ra y^*)^*)\ra (x\ra y^*)^*)\ra (x\ra y^*)^*=x$, \\
$\hspace*{2.00cm}$ $(((x\ra y^*)\ra x)\ra (x\ra y^*)^*)\ra (x\ra y^*)^*=x$, \\
$\hspace*{2.00cm}$ $(x\ra (x\ra y^*)^*)\ra (x\ra y^*)^*=x$ (by $(Impl)$), \\
$\hspace*{2.00cm}$ $x\Cup (x\ra y^*)^*=x$, \\
$\hspace*{2.00cm}$ $x^*\Cap (x\ra y^*)=x^*$, \\
$\hspace*{2.00cm}$ $x\Cap (x^*\ra y)=x$ (replacing $x$ by $x^*$ and $y$ by $y^*$). \\
Thus $X$ satisfies condition $(IOM)$, so that, by Theorem \ref{ioml-60}, $X$ verifies axiom $(QW_2)$. 
\end{proof} 

\begin{theorem} \label{gioml-60} Let $(X,\ra,^*,1)$ be an involutive BE algebra satisfying condition $(iG)$. 
Then axioms $(QW_1)$ and $(QW_2)$ are equivalent.  
\end{theorem}
\begin{proof} It follows by Propositions \ref{gioml-40} and \ref{gioml-50}. 
\end{proof}

\begin{corollary} \label{gioml-70}
If $X$ is an involutive BE algebra satisfying conditiom $(iG)$, the following are equivalent: \\
$(a)$ $X$ satisfies axiom $(QW_1);$ \\
$(b)$ $X$ satisfies axiom $(QW_2);$ \\
$(c)$ $X$ satisfies axiom $(QW)$. 
\end{corollary}
\begin{proof} 
By Theorem \ref{gioml-60}, $(a)\Leftrightarrow (b)$. 
If $X$ satisfies $(QW_1)$, then by Proposition \ref{gioml-50}, $X$ verifies $(QW_2)$, so that $(a)\Rightarrow (c)$. 
Similarly, if $X$ satisfies $(QW_2)$, then by Proposition \ref{gioml-40}, $X$ verifies $(QW_1)$, hence 
$(b)\Rightarrow (c)$. 
Since by definition, $(QW)$ implies $(QW_1)$ and $(QW_2)$, it follows that $(c)\Rightarrow (a)$ and  
$(c)\Rightarrow (b)$, and the proof is complete. 
\end{proof}

\begin{corollary} \label{gioml-80}
Denoting by $\mathbf{IOMSL}$ the class of implicative-orthomodular softlattices, we have: \\
$\hspace*{4.00cm}$ $\mathbf{IOMSL}=\mathbf{preW}\bigcap \mathbf{IOSL}$ and \\
$\hspace*{4.00cm}$ $\mathbf{IOMSL}=\mathbf{QW}\bigcap \mathbf{IOSL}$.
\end{corollary}

\begin{theorem} \label{gioml-100}
The implicative-orthomodular softlattices are equivalent to implicative-orthomodular lattices. 
\end{theorem}
\begin{proof}
Let $X$ be an implicative-orthomodular softlattice, that is an involutive BE algebra satisfying conditions $(QW_2)$ 
and $(iG)$. By Proposition \ref{gioml-90}, $X$ satisfies condition $(Impl)$, hence it is an implicative-orthomodular lattice. Conversely, if $X$ is an implicative-orthomodular lattice, it satisfies $(QW_2)$ and $(Impl)$. 
Since by Lemma \ref{iol-30}, condition $(Impl)$ implies condition $(iG)$, it follows that $X$ is an implicative-orthomodular softlattice.  
\end{proof}

\begin{lemma} \label{gioml-150} Let $(X,\ra,^*,1)$ be an involutive BE algebra satisfying axiom $(QW_2)$.  
The following hold for all $x,y,z\in X$: \\ 
$(1)$ $(y\ra (x\ra y^*)^*)\ra (z\ra (x\ra y^*)^*)=z\ra y;$ \\     
$(2)$ $((x\ra y)\ra y)\ra ((x\ra y)\ra z)=y\ra z;$ \\  
$(3)$ $(x\ra (y\ra z^*))\ra (x\ra (y\ra (x\ra (y\ra z^*))^*))^*=(x\ra y^*)^*;$ \\  
$(4)$ $(x\ra y^*)\ra (z\ra (x\ra (x\ra y^*)^*)^*)=z\ra x;$ \\  
$(5)$ $((y\ra z^*)\ra (x\ra (y\ra (y\ra z^*)^*)^*))\ra $ \\
$\hspace*{1.50cm}$ $((y\ra z^*)\ra (x\ra ((y\ra z^*)\ra (x\ra (y\ra (y\ra z^*)^*))^*)))^*
                            =((y\ra z^*)\ra x^*)^*$. \\   
$(6)$ $(y\ra z^*)\ra (x\ra (y\ra (y\ra z^*)^*)^*)=x\ra y$. \\  
$(7)$ $ (x\ra y)\ra ((y\ra z^*)\ra (x\ra (x\ra y)^*))^*=((y\ra z^*)\ra x^*)^*$. 
\end{lemma}
\begin{proof}
Let $(X,\ra,^*,1)$ be an involutive BE algebra satisfying axiom $(QW_2)$. \\
$(1)$ We have succesively: \\ 
$\hspace*{2cm}$ $y\Cup (y\ra x^*)^*=y$ (by $(IOM^{''})$). \\
$\hspace*{2cm}$ $y\Cup (x\ra y^*)^*=y$ (Lemma \ref{qbe-10}$(6)$) \\
$\hspace*{2cm}$ $z\ra (y\Cup (x\ra y^*)^*)=z\ra y$. \\
$\hspace*{2cm}$ $(y\ra (x\ra y^*)^*)\ra (z\ra (x\ra y^*)^*)=z\ra y$ (by $(BE_4)$). \\
$(2)$ From $(1)$, using Lemma \ref{qbe-10}$(3)$, we get $((x\ra y^*)\ra y^*)\ra ((x\ra y^*)\ra z^*)=y^*\ra z^*$. 
Replacing $y$ by $y^*$ and $z$ by $z^*$, it follows that $((x\ra y)\ra y)\ra ((x\ra y)\ra z)=y\ra z$. \\
$(3)$ We have succesively: \\
$\hspace*{2.00cm}$ $(x\ra y^*)^*\Cup ((x\ra y^*)^*\ra z^*)^*=(x\ra y^*)^*$ (by $(IOM^{''})$). \\
$\hspace*{2.00cm}$ $((x\ra y^*)^*\ra ((x\ra y^*)^*\ra z^*)^*)\ra ((x\ra y^*)^*\ra z^*)^*=(x\ra y^*)^*$. \\
$\hspace*{2.00cm}$ $((x\ra y^*)^*\ra z^*)\ra ((x\ra y^*)^*\ra ((x\ra y^*)^*\ra z^*)^*)^*=(x\ra y^*)^*$ \\ 
                                                          $\hspace*{10.00cm}$ (Lemma \ref{qbe-10}$(3)$). \\
$\hspace*{2.00cm}$ $(x\ra (y\ra z^*))\ra (x\ra (y\ra (x\ra (y\ra z^*))^*))^*=(x\ra y^*)^*$ \\
                                                          $\hspace*{10.00cm}$ (Lemma \ref{qbe-10}$(7)$). \\
$(4)$ Following the method used in $(3)$, we get succesively: \\
$\hspace*{2.00cm}$ $x\Cup (x\ra y^*)^*=x$ ($(IOM^{''})$). \\ 
$\hspace*{2.00cm}$ $z\ra (x\Cup (x\ra y^*)^*)=z\ra x$. \\
$\hspace*{2.00cm}$ $z\ra ((x\ra (x\ra y^*)^*)\ra (x\ra y^*)^*)=z\ra x$. \\
$\hspace*{2.00cm}$ $z\ra ((x\ra y^*)\ra (x\ra (x\ra y^*)^*)^*)=z\ra x$ (Lemma \ref{qbe-10}$(5)$). \\ 
$\hspace*{2.00cm}$ $(x\ra y^*)\ra (z\ra (x\ra (x\ra y^*)^*)^*)=z\ra x$ (by $(BE_4)$). \\
$(5)$ It follows from $(3)$, replacing $x$ by $y\ra z^*$, $y$ by $x$ and $z$ by $y\ra (y\ra z^*)^*$. \\ 
$(6)$ It follows from $(4)$, replacing $x$ by $y$, $y$ by $z$ and $z$ by $x$. \\
$(7)$ It follows from $(5)$ taking into consideration $(6)$. 
\end{proof}

\begin{lemma} \label{gioml-160} Let $(X,\ra,^*,1)$ be an involutive BE algebra satisfying axioms $(QW_2)$ and  $(Iabs$-$i)$. The following hold for all $x,y,z\in X$: \\ 
$(1)$ $((x\ra y^*)\ra (z\ra y))\ra (x\ra y^*)=x\ra y^*;$ \\            
$(2)$ $x\ra (y\ra (y^*\ra ((x\ra y^*)\ra z^*))^*)=x\ra y^*;$ \\        
$(3)$ $x\ra (y^*\ra (y\ra ((x\ra y)\ra ((y\ra z)^*\ra (x\ra (x\ra y)^*)))))=x\ra y;$ \\  
$(4)$ $x\ra (y^*\ra (y\ra ((y\ra z^*)\ra x^*)^*)^*)=x\ra y;$ \\         
$(5)$ $(x\ra (x\ra y^*)^*)\ra (y^*\ra (y\ra ((y\ra x^*)\ra (x\ra (x\ra y^*)^*)^*)^*)^*)$ \\
      $\hspace*{3.50cm}$ $=(x\ra (x\ra y^*))\ra y;$ \\                
$(6)$ $(x\ra (x\ra y^*)^*)\ra (y^*\ra (y\ra x^*)^*)=(x\ra (x\ra y^*)^*)\ra y$.  
\end{lemma}
\begin{proof}
Let $X$ be an involutive BE algebra satisfying axioms $(QW_2)$ and $(Iabs$-$i)$. By Theorem \ref{ioml-60}, 
X verifies the equivalent conditions $(IOM)$, $(IOM^{'})$, $(IOM^{''})$. \\
$(1)$ Since by $(Iabs$-$i)$, $(x\ra (x\ra y))\ra x=x$, replacing $x$ by $x\ra y^*$ and 
$y$ by $(y\ra (x\ra y^*)^*)\ra z^*$, 
we get succesively: \\
$\hspace*{1.00cm}$ $((x\ra y^*)\ra ((x\ra y^*)\ra ((y\ra (x\ra y^*)^*)\ra z^*)))\ra (x\ra y^*)=x\ra y^*$. \\
$\hspace*{1.00cm}$ $((x\ra y^*)\ra ((y\ra (x\ra y^*)^*)\ra ((x\ra y^*)\ra z^*)))\ra (x\ra y^*)=x\ra y^*$ 
                                                               $\hspace*{10.00cm}$ (by $(BE_4)$). \\
$\hspace*{1.00cm}$ $((x\ra y^*)\ra ((y\ra (x\ra y^*)^*)\ra (z\ra (x\ra y^*))))\ra (x\ra y^*)=x\ra y^*$ 
                                                              $\hspace*{10.00cm}$ (Lemma \ref{qbe-10}$(3)$). \\
$\hspace*{1.00cm}$ $((x\ra y^*)\ra (z\ra y))\ra (x\ra y^*)=x\ra y^*$ (Lemma \ref{gioml-150}$(1)$). \\
$(2)$ We have the following: \\
$\hspace*{1.00cm}$ $((x\ra y^*)\ra (z\ra y))\ra (x\ra y^*)=x\ra y^*$ (by $(1)$). \\
$\hspace*{1.00cm}$ $(x\ra y^*)^*\ra ((x\ra y^*)\ra (z\ra y))^*$ (Lemma \ref{qbe-10}$(6))$). \\
$\hspace*{1.00cm}$ $x\ra (y\ra ((x\ra y^*)\ra (z\ra y))^*)=x\ra y^*$ (Lemma \ref{qbe-10}$(7)$). \\
$\hspace*{1.00cm}$ $x\ra (y\ra ((x\ra y^*)\ra (y^*\ra z^*))^*)=x\ra y^*$ (Lemma \ref{qbe-10}$(6)$). \\
$\hspace*{1.00cm}$ $x\ra (y\ra ((y^*\ra z^*)^*\ra (x\ra y^*)^*)^*)=x\ra y^*$ (Lemma \ref{qbe-10}$(6)$). \\ 
$\hspace*{1.00cm}$ $x\ra (y\ra (y^*\ra (z\ra (x\ra y^*)^*))^*)=x\ra y^*$ (Lemma \ref{qbe-10}$(7)$). \\
$\hspace*{1.00cm}$ $x\ra (y\ra (y^*\ra ((x\ra y^*)\ra z^*))^*)=x\ra y^*$ (Lemma \ref{qbe-10}$(3)$). \\
$(3)$ It follows from $(2)$, replacing $y$ by $y^*$ and $z$ by $(y\ra z^*)\ra (x\ra (x\ra y)^*)$. \\
$(4)$ It follows from $(3)$, applying Lemma \ref{gioml-150}$(7)$. \\
$(5)$ It follows from $(4)$, replacing $x$ by $x\ra (x\ra y^*)^*$ and $z$ by $x$. \\
$(6)$ Since by $(IOM^{''})$, $x\Cup (x\ra y^*)^*)^*=x$, we get $(x\ra (x\ra y^*)^*)\ra (x\ra y^*)^*=x$. 
Using Lemma \ref{qbe-10}$(6)$, we get $(y\ra x^*)\ra (x\ra (x\ra y^*)^*)=x$, so that $(5)$ implies $(6)$.  
\end{proof}

\begin{proposition} \label{gioml-170} Let $(X,\ra,^*,1)$ be an involutive BE algebra. 
If $X$ satisfies axioms $(QW_2)$ and $(Iabs$-$i)$, then $X$ satisfies axiom $(QW_1)$. 
\end{proposition}
\begin{proof}
Let $(X,\ra,^*,1)$ be an involutive BE algebra satisfying axioms $(QW_2)$ and $(Iabs$-$i)$, and let $x,y\in X$. 
We have the following equivalences: \\
$\hspace*{2.00cm}$ $(x\ra (x\ra y^*)^*)\ra (y^*\ra (y\ra x^*)^*)=(x\ra (x\ra y^*)^*)\ra y$ 
                                                   (Lemma \ref{gioml-160}$(6)$). \\ 
$\hspace*{2.00cm}$ $y^*\ra ((x\ra (x\ra y^*)^*)\ra (x\ra y^*)^*)=(x\ra (x\ra y^*)^*)\ra y$ (by $(BE_4)$). \\
$\hspace*{2.00cm}$ $y^*\ra (x\Cup (x\ra y^*)^*)=(x\ra (x\ra y^*)^*\ra y$. \\                                 
$\hspace*{2.00cm}$ $y^*\ra x=(x\ra (x\ra y^*)^*\ra y$ (by $(IOM^{''})$). \\
$\hspace*{2.00cm}$ $y^*\ra (x\ra (x\ra y^*)^*)^*=y^*\ra x$ (Lemma \ref{qbe-10}$(3)$). \\
$\hspace*{2.00cm}$ $x\ra (y\ra (y\ra x)^*)^*=x\ra y$ (replacing $y$ by $x^*$ and $x$ by $y$). \\
$\hspace*{2.00cm}$ $x\ra ((x^*\ra y^*)\ra y^*)^*=x\ra y$ (Lemma \ref{qbe-10}$(6)$). \\
$\hspace*{2.00cm}$ $x\ra (x\Cap y)=x\ra y$. \\
Hence $X$ satisfies axiom $(QW_1)$.  
\end{proof}

\begin{theorem} \label{gioml-180} 
The implicative-orthomodular widelattices coincide to quantum-Wajsberg algebras verifying $(Iabs$-$i)$.  
\end{theorem}
\begin{proof}
By definition, an implicative-orthomodular widelattice $X$ is an involutive BE algebra verifying axioms $(QW_2)$ and $(Iabs$-$i)$. According to Proposition \ref{gioml-170}, $X$ satisfies axiom $(QW_1)$, so that it is a 
quantum-Wajsberg algebra satisfying $(Iabs$-$i)$. 
Conversely, a quantum-Wajsberg algebra satisfying $(Iabs$-$i)$ verifies $(QW_2)$, so that it is an implicative-orthomodular widelattice. 
\end{proof} 

\begin{corollary} \label{gioml-190} 
Since by \cite{Ciu78} a quantum-Wajsberg algebra satisfies condition $(QW_3)$, it follows that 
any implicative-orthomodular widelattice is a meta-Wajsberg algebra verifying $(Iabs$-$i)$.  
\end{corollary}

\begin{corollary} \label{gioml-200}
Denoting by $\mathbf{IOMWL}$ the class of implicative-orthomodular widelattices, we have: \\
$\hspace*{4.00cm}$ $\mathbf{IOMWL}=\mathbf{QW}\bigcap \mathbf{IOWL}$ and \\
$\hspace*{4.00cm}$ $\mathbf{IOMWL}\subset \mathbf{metaW}$. 
\end{corollary}

The nexe example follows from \cite[Ex. 13.3.3]{Ior35}. 

\begin{example} \label{gioml-210} 
Let $X=\{0,a,b,c,d,1\}$ and let $(X,\ra,^*,1)$ be the involutive BE algebra 
with $\ra$ and the corresponding operation $\Cap$ given in the following tables:  
\[
\begin{array}{c|ccccccc}
\ra & 0 & a & b & c & d & 1 \\ \hline
0   & 1 & 1 & 1 & 1 & 1 & 1 \\ 
a   & d & 1 & 1 & 1 & d & 1 \\ 
b   & c & 1 & 1 & 1 & 1 & 1 \\ 
c   & b & 1 & 1 & 1 & 1 & 1 \\
d   & a & a & 1 & 1 & 1 & 1 \\
1   & 0 & a & b & c & d & 1
\end{array}
\hspace{10mm}
\begin{array}{c|ccccccc}
\Cap & 0 & a & b & c & d & 1 \\ \hline
0    & 0 & 0 & 0 & 0 & 0 & 0 \\ 
a    & 0 & a & b & c & 0 & a \\ 
b    & 0 & a & b & c & d & b \\ 
c    & 0 & a & b & c & d & c \\
d    & 0 & 0 & b & c & d & d \\
1    & 0 & a & b & c & d & 1
\end{array}
.
\]

We can check that $X$ is an implicative-orthomodular widelattice. 
Since $(b\ra 0)\ra b=1\neq b$, it follows that axiom $(Impl)$ is not verified, so that $X$ is not an 
implicative-orthomodular lattice. 
\end{example}

$\vspace*{5mm}$

\section{Concluding remarks}

We introduced and studied in \cite{Ciu78, Ciu79, Ciu80} the quantum-Wajsberg algebras as implicative counterpart 
of quantum-MV algebras defined by R. Giuntini in \cite{Giunt2}. 
In the papers \cite{Ciu81, Ciu82}, we gave generalizations of quantum-Wajsberg algebras:  
implicative-orthomodular, pre-Wajsberg and meta-Wajsberg algebras. \\
In this paper, we introduce and study certain subclasses of quantum-Wajsberg algebras.  
We redefine the orthomodular lattices, by introducing the notion of implicative-orthomodular lattices, giving 
certain characterizations of these structures and proving that they are quantum-Wajsberg algebras. 
We also define and study the implicative-modular algebras and we prove that they are implicative-orthomodular lattices. 
We also introduce the implicative-orthomodular softlattices and widelattices as generalizations of implicative-orthomodular lattices. 
Finally,we prove that the implicative-orthomodular softlattices are equivalent to implicative-orthomodular lattices, and that  the implicative-orthomodular widelattices are special cases of quantum-Wajsberg algebras. \\
As we mentioned, according to \cite[Th. 2.3.9]{DvPu} every orthomodular lattice determines a quantum-MV 
algebra by taking $\oplus$ as the supremum $\vee$ and $^*$ as the orthocomplement $^{'}$. 
In the proof of this result, the authors used the Sasaki projection and the Foulis-Holland theorem for orthomodular 
lattices. 
As a future topic of research, the Sasaki projection and the Foulis-Holland theorem could be reformulated for 
the case of implicative-orthomodular lattices, and the above mentioned theorem could be adapted accordingly.

          

$\vspace*{1mm}$

\end{document}